\documentclass{amsart}

\usepackage{amssymb,latexsym,amsmath,amscd,epsfig,amsthm, amsfonts, stmaryrd}
\theoremstyle{theorem}
\newtheorem{lemma}{Lemma}
\newtheorem{theorem}{Theorem}
\newtheorem{corollary}{Corollary}

\newtheorem{proposition}{Proposition}

\theoremstyle{remark}

\newcommand{\R}{{\mathbb R}}
\newcommand{\N}{{\mathbb N}}
\newcommand{\C}{{\mathbb C}}
\newcommand{\Z}{{\mathbb Z}}

\usepackage{graphicx}
\usepackage{float}
\DeclareGraphicsExtensions{.pdf,.png,.jpg} 
\begin{document}

\title[Unique continuation for discrete harmonic functions]{A note on unique continuation for discrete harmonic functions}
\author{Maru Guadie and Eugenia Malinnikova
}  
\address{Department of Mathematical Sciences,
\newline
Norwegian University of Science and Technology,
\newline
7491, Trondheim, Norway}
\email{guadie@math.ntnu.no, eugenia@math.ntnu.no
}

\begin{abstract} 
We suggest an elementary quantitative unique continuation argument for harmonic functions that can be adapted to the discrete case of harmonic functions on the lattice. The analog of the three balls theorem for discrete harmonic functions that we obtain contains an additional term that depends on the mesh size of the lattice and goes to zero when the mesh size goes to zero. We also show that any discrete harmonic function on a cube can be extended to the whole lattice by a discrete harmonic polynomial.
\end{abstract}
\maketitle

\section{Introduction} 
Quantitative unique continuation is an important tool in the study of solutions of elliptic and parabolic problems. It has many applications, including stability estimates for the Cauchy problem, see \cite{ARRV}. The simplest quantitative unique continuation statement is the three balls theorem. For classical harmonic functions it follows from logarithmic convexity of the $L^2$-norms, that in turn is obtained using the rotational symmetry and ellipticity of the Laplace operator and can be proved by expansions in eigenfunctions of the Laplace-Beltrami operator on the sphere \cite{KM}. However, the logarithmic convexity can be generalized to general elliptic equations and it has been successfully used in unique continuation problems, \cite{A, GL}.

Situation with unique continuation changes drastically when one considers discrete models for elliptic equations. It is easy to construct a discrete harmonic function (even a discrete harmonic polynomial, see the last section of this note) on the lattice $\Z^n$  that vanishes on a large cube of the lattice without being zero identically. On the other hand  some version of unique continuation should hold at least when we fix the domains and let the mesh size of the lattice go to zero, since discrete harmonic functions then approximate continuous ones. One of the obstacles for an elementary estimate similar to logarithmic convexity for continuous harmonic functions is that discrete Laplacian is not rotationally symmetric.

 Recently quantitative uniqueness (from Cauchy data)  for discrete models of elliptic PDEs were obtained by Carleman type inequalities, we refer the reader to  \cite{BHL,EG} for motivation and interesting results; we mention also an earlier work \cite{KS} that contains discrete Carleman estimates. It is known that the mesh size of the discretization appears in propagation of smallness inequalities obtained by this method. The  Carleman inequality is the most common tool  for quantitative unique continuation due to its flexibility. It can be adjusted to quite general setting. However even for the simplest case of the five-point discrete Laplacian on $\Z^2$ the quantitative uniqueness by Carleman estimates is very technical. We suggest another approach that is based on the analyticity of the Poisson kernel and thus can be applied only for the case of good equations but is simple and gives direct constructive estimates for discrete harmonic functions on $n$-dimensional lattices.   

We consider the standard lattice $(h\Z)^n$ in $n$-dimensional space $\R^n$, we always assume that $N=h^{-1}$ is a positive integer. A function $u:(h\Z)^n\rightarrow\R$ is called $h$-discrete harmonic at a point 
$x\in (h\Z)^n$ if 
\[2n u(x)=\sum_{j=1}^n u(x+he_j)+u(x-he_j),\]
where $\{e_j\}_{j=1}^n$ is the standard orthonormal basis for $\R^n$. Different logarithmic convexity estimates for the norms of discrete harmonic functions, with norms taken over parallel segments or parallel lines, can be found in \cite{FM, RHH, G, GM}.

In this note we obtain an analog of the three balls theorem for discrete harmonic functions.
We define by $Q_d$ the cube $[-d,d]^n\subset\R^n$ and by $Q_d^h$ its discretization, $Q_d^h=Q_d\cap(h\Z)^n$. More generally for any set $E\subset\R^n$ we denote $E^h=E\cap(h\Z)^n$. Then our main result is the following
\begin{theorem}\label{th:m}
Suppose that $r<R<1$. There exist positive constants $C, N_0, \delta, \alpha$ that depend on $r,R$ with $\alpha,\delta<1$ such that for any $h=N^{-1}$, $N\in \N$, $N>N_0$ and any
$h$-discrete harmonic function $u$ in $Q_1^h$ that satisfies $\max_{Q_r^h}|u(x)|\le\varepsilon$ and $\max_{Q_1^h}|u(x)|\le M$ the following inequality holds
\[
\max_{Q_R^h}|u(x)|\le C(\varepsilon^\alpha M^{1-\alpha}+\delta^{\sqrt{N}}M).\]
\end{theorem} 

The text is organized in the following way. In the next section we give a new proof of the three sphere theorem for continuous harmonic functions that can be adjusted to the discrete case. The adjustment is done in Section 3, where we prove the main result. For the case $R<2r<2^{-2n-3}$ we suggest relatively simple formulas for $\alpha$ and $\delta$, for the general case one has to iterate the estimate in a standard way. In the last Section we prove that any discrete harmonic function on a cube can be extended to the whole lattice by  a discrete harmonic polynomial.

\section{Continuous case}
The proposition below is the three balls theorem for harmonic functions. It is well known, the standard approach is to prove logarithmic convexity for $L^2$-norms and then use elliptic estimates to obtain $L^\infty$ estimates, see \cite{KM} for the details. We give another elementary proof that will be extended to discrete situation in the next section. We work in $n$-dimensional Euclidean space $\R^n$ and fix $n$, so our constants may depend on the dimension.
\begin{proposition}\label{pr:cont}
Let  $0<r<R<1/4$. There exist constants $C>0$ and $\alpha\in(0,1)$ such that for any harmonic function $u$ in the unit ball with 
\[\max_{|x|\le r}|u(x)|=\varepsilon,\quad \max_{|x|\le 1}|u(x)|=M,\]
the following inequality holds
\begin{equation}\label{eq:cth}
\max_{|x|=R}|u(x)|\le C\varepsilon^\alpha M^{1-\alpha}.
\end{equation}
\end{proposition}

\begin{proof}
We have 
\[
u(x)=\int_{S^{n-1}}P(x,y)u(y)d\sigma(y),\]
where $P(x,y)=\gamma_n(1-|x|^2)(|x-y|)^{-n}$ is the standard Poisson kernel for the unit ball.
We fix a point $x_0$ such that $|x_0|\le R$. The idea of the proof is to approximate $P(x_0,y)$ by a linear combination of the form $\sum_{k=1}^mc_kP(x_k,y)$ with $|x_k|\le r$. We will need two estimates, one for the error $r_m(x_0,y)$ of the approximation and another for the sum of the absolute values of the coefficients of the approximating linear combination.

We choose points $x_k$ on the segment $[0,rR^{-1}x_0]$,  $x_k=t_krR^{-1}x_0$, $t_k\in(0,1)$, and consider the standard Lagrange interpolation of the function $f(t)=P(trR^{-1}x_0,y)$, then
\[
c_k=\prod_{j\neq k}\frac{r^{-1}R-t_j}{t_k-t_j}.\]
Consider the polynomial $H_m(t)=(t-t_1)...(t-t_m)$, we have \[|c_k|\le (R/r)^m|H_m'(t_k)|^{-1}.\]
Now we choose $t_1,...,t_m$ to be the Chebyshev nodes,  $t_k=\cos\left(\pi\frac{2k-1}{2m}\right)$, then $H_m(t)=2^{1-m}T_m(t)$, where $T_m$ is the Chebyshev polynomial of the first kind. We have $H'_m(t)=m2^{1-m}U_{m-1}(t)$, where $U_{m-1}$ is the Chebyshev polynomial of the second kind, see for example \cite[Chapter 2]{DB}. Therefore
\[
U_{m-1}(t_k)=U_{m-1}\left(\cos\left(\pi\frac{2k-1}{2m}\right)\right)=\frac{\sin \left(\pi\frac{2k-1}{2}\right)}{\sin \left(\pi\frac{2k-1}{2m}\right)}=\frac{(-1)^{k-1}}{\sin \left(\pi\frac{2k-1}{2m}\right)}.\]
Then $|H_m'(t)|\ge m2^{1-m}$ and $|c_k|\le m^{-1}(2R/r)^m$.

In order to estimate the error of the approximation, we use an analytic extension of the function $f(t)=P(trR^{-1}x_0,y)$  to the disk of radius $1/2r$ centered at the origin on the complex plane, see for example \cite[Chapter 4]{DB} for the residue method in the interpolation error estimate. We have
\[f(z)=\gamma_n\frac{1-r^2R^{-2}|x_0|^2z^2}{\left(\sum_j (rR^{-1}x_{0,j}z-y_j)^2\right)^{n/2}},\]
 this extension is bounded by a constant $A_n$. We consider the  function 
\[\Omega(z)=\frac{ f(z)H_m(R/r)}{(z-R/r)H_m(z)}\]
which is meromorphic in $\{|z|<1/2r\}$ and has simple poles at the points $R/r$ and $t_1,...,t_m$. Then by the residue theorem, we get (see also \cite[Theorem 4.3.3]{DB}),
\begin{multline*}
|r_m(x_0,y)|=|P(x_0,y)-\sum_{k=1}^mc_kP(x_k,y)|=|f(R/r)-\sum_{k=1}^m c_kf(t_k)|=\\\left|\frac{1}{2\pi i}\int_{|z|=1/2r}\frac{f(z)H_m(R/r)}{(z-R/r)H_m(z)}dz\right|\le \frac{A_n}{1-2R}\left(\frac{2R}{1-2r}\right)^m.
\end{multline*} 
Thus we have the following two estimates
\[
\sum_{k=1}^m |c_k|\le \left(\frac{2R}{r}\right)^m= B^m\quad \text{and}\quad |r_m(x_0,y)|\le aq^m,\]
for some $B,a,q $ such that $q<1$.
Then
\begin{multline}\label{eq:cont}
|u(x_0)|=\left|\int_{S^{n-1}}P(x_0,y)u(y)d\sigma(y)\right|\le\\
\sum_{k=1}^m|c_k|\left|\int_{S^{n-1}}P(x_k,y)u(y)d\sigma(y)\right|+\left|\int_{S^{n-1}}r_m(x_0,y)u(y)d\sigma(y)\right|\le\\
\sum_{k=1}^m|c_k||u(x_k)|+c_n\max_{|y|=1}|r_m(x_0,y)||u(y)|\le B^m\varepsilon+a_1q^m M.
\end{multline}
To minimize the sum we choose $m=\left[(\log M-\log\varepsilon)\left(\log B-\log q\right)^{-1}\right]+1$, where $[t]$ is the largest integer less than or equal to $t$, and obtain the required inequality (\ref{eq:cth}) with $C=B+a_1=2Rr^{-1}+A_nc_n(1-R)^{-1}$ and $\alpha=1-\log B(\log B-\log q)^{-1}$.
\end{proof}
The Chebyshev nodes is the standard choice in the interpolation problems, when we want to have a good control over the coefficients. More generally the Fekete points of a given compact  set $K\subset\R$ can be chosen. They appear also in the quantitative propagation of smallness from the sets of positive capacity, see \cite{M}.

\section{Discrete case}
We start by the following discrete version of the Poisson integral representation
\begin{equation}\label{eq:Pk}
u(x)=\sum_{y\in \partial Q_1^h} u(y)P_h(x,y),\end{equation}
where for each $y\in\partial Q_1^h$, the function $P_h(x,y)$ is $h$-discrete  harmonic in the variable $x$ in $Q_1^h$, and satisfies the boundary conditions $P_h(y,y)=1$ and $P_h(z,y)=0$ for any $z\in\partial Q^h_1\setminus\{y\}$. 

We will write down an analytic expression for $P_h(x,y)$. Note that since we consider discrete function with finitely many values, its analytic extension is not unique. Without loss of generality, we assume that $y=(y_1,...,y_{n-1},1)$. For each $K=(k_1,...,k_{n-1})\in ((0,2N)\cap \Z)^{n-1}=\llbracket1,2N-1\rrbracket^{n-1}$ we define $a^h_K$ to be the only positive solution of the equation
\[
\cosh \frac{h a^h_K}{2}=n-\sum_{j=1}^{n-1}\cos \frac{\pi k_jh}{2}.\]
Then 
\[
f^h_K(x)=\sinh( a^h_K (x_n+1)/2)\prod_{j=1}^{n-1}\sin (\pi k_j (x_j+1)/2)\]
is $h$-discrete harmonic and vanishes on all sides of the cube except the one where $y$ lies. It is easy to check that
\[
P_h(x,y)=\left(\frac{1}{N}\right)^{n-1}\sum_{K} \prod_{j=1}^{n-1}\sin \left(\pi k_j \frac{x_j+1}{2}\right)\sin \left(\pi k_j \frac{y_j+1}{2}\right)\frac{\sinh \left(a_K^h \frac{x_n+1}{2}\right)}{\sinh a_K^h},\]
where summation is taken over $K\in\llbracket 1,2N-1\rrbracket^{n-1}$. This function is $h$-discrete harmonic in $Q_1^h$ and satisfies the required boundary conditions.

\begin{proposition}\label{pr:ext}
For any $y\in\partial Q_1^h$ and $(x_1,..,x_{j-1}, x_{j+1},...,x_n)\in[-1/2,1/2]^{n-1}$, $j=1,...,n$,  the function $f(t)=P_h(x_1,..., x_{j-1},t,x_{j+1},...,x_n,y)$ has a holomorphic extension to the domain $\Omega=\{z: -1/2\le \Re z\le 1/2, -1/16\le \Im z\le 1/16\}\subset \C$ that satisfies $|f(z)|\le CN^{1-n}$ for any $z\in\Omega$.
\end{proposition}

\begin{proof} The holomorphic extension is given by the formula above. We need to prove the estimate.
First, we note that either $ha_K^h\ge 2$ or
\[
\left(\frac{ha_K^h}{2}\right)^2\ge \cosh \frac{ha_K^h}{2}-1=\sum_{j=1}^{n-1}(1-\cos \frac{\pi k_j h}{2})\ge \frac14\sum_{j=1}^{n-1}  k_j^2h^2.\]
Thus either $a_K^h\ge 2N$ or $a_K^h\ge \|K\|$, where  $\|K\|^2=\sum_j k_j^2\ge n^{-1}\left(\sum_jk_j\right)^2$. 

We consider two cases $j=n$ and $j\neq n$. First, if $j=n$ then, since  $|\Re z|\le 1/2$, we have
\begin{multline*}
|f(z)|\le C N^{1-n}\sum_{K} \exp(-a_K^h/4)\le\\
CN^{1-n}\sum_K\exp(-\|K\|/4)+CN^{1-n}(2N)^n\exp(-N/2)\le\\
CN^{1-n}\left(\sum_{k=1}^\infty\exp\left(-\frac{k}{4\sqrt{n}}\right)\right)^{n-1}+CN^{1-n}(2N)^n\exp(-N/2)\le C_nN^{1-n}.
\end{multline*}
Otherwise, if $j\neq n$, and taking into account that $|\Im z|\le 1/16$, we obtain
\begin{multline*}
|f(z)|\le CN^{1-n}\sum_K \exp(\pi k_j/32-a_K^h/4)\le\\
 CN^{1-n}\sum_K \exp(-a_K^h/32)+CN^{1-n}N^n\exp(-N/4+\pi N/16)\le C_nN^{1-n}.
\end{multline*}

\end{proof}

We also need a discrete version of Chebyshev's nodes.

\begin{lemma} \label{l:dCh}
Suppose that $M>m^2$. Then there exists a polynomial $H_{m,M}(t)=(t-s_1)...(t-s_m)$, where $s_j\in M^{-1}\Z\cap[-1,1]$, such that $|H_{m,M}'(s_j)|\ge m2^{1-m}$ for any $j=1,...,m$.
\end{lemma}

\begin{proof}
Let $t_k=\cos ((2k-1)\pi/2m)$ be the classical Chebyshev nodes, an elementary estimate shows that $|t_j-t_k|\ge m^{-2}$ when $j\neq k$. We choose $s_j\in M^{-1}\Z$ such that  $|s_j-2t_j+1|\le  (2M)^{-1}$. Then 
\[
|s_j-s_k|\ge 2|t_j-t_k|-M^{-1}\ge |t_j-t_k|.\]
We have
\[
|H_{m,M}'(s_j)|=\prod_{k\neq j}|s_j-s_k|\ge \prod_{k\neq j}|t_j-t_k|\ge m2^{1-m}.\]
\end{proof}

Combining the statements above and repeating the argument from the previous section, we obtain the following
\begin{lemma}\label{l:three}
Suppose that $r<R<2r<2^{-2n-3}$. There exist constants $A,B,q$ that depend on $r,R$ with $q<1$ such that for any
$h$-discrete harmonic function $u$ in $[-1,1]^n$ and any $m<\sqrt{rh^{-1}}$ we have
\[
\max_{Q_R^h}|u(x)|\le A(B^m\max_{Q_r^h}|u(x)|+q^m\max_{Q_1^h}|u(x)|).\]
\end{lemma}

\begin{proof} We may assume that $r,R\in h\Z$. 
We consider the following chain of rectangles $R_0=[-r,r]^n, R_1=[-R,R]\times[-r,r]^{n-1}$,..., $R_n=[-R,R]^n$. We want to prolongate the estimate from $R_j$ to $R_{j+1}$. Let $x=(x_1,...,x_{j},...,x_n)\in \partial R_j\setminus R_{j-1}$, then $r<|x_j|\le R$. For each $y\in \partial Q^h_1$ we consider the Poisson kernel $P_h(x,y)$ as a function of $x_j$. More precisely, we fix $y\in\partial Q^h_1$ and define
\[f(t)=P_h(x_1,...,tx_j|x_j|^{-1}r,..., x_n,y).\]
Further, by Proposition \ref{pr:ext}, $f$ can be extended to a holomorphic function in the domain $D=\{z\in \C: |\Re z|\le (2r)^{-1}, |\Im z|\le (16r)^{-1}\}$, where it satisfies $|f(z)|\le CN^{1-n}$. 

We let $M=Nr$ and choose $s_1,...,s_m$ as in Lemma \ref{l:dCh}. Applying the Lagrange interpolation with nodes $s_j$ we approximate $f(R/r)$ by $\sum_k c_kf(s_k)$, where 
\[
c_k=\prod_{j\neq k}\frac{r^{-1}R-s_j}{s_k-s_j}.\]
By Lemma \ref{l:dCh} we get $|c_k|\le(r^{-1}R+1)^m2^{m-1}m^{-1}$. then
\[
\sum_k|c_k|\le  \left(\frac{2(r+R)}{r}\right)^m.\]
The error of the approximation is
\begin{multline*}
|r_j(x,y)|=|P(x,y)-\sum_k c_kP(x_k,y)|=\\
\left|\frac{1}{2\pi}\int_{\partial D}\frac{f(z)H_{m,M}(R/r)}{(z-R/r)H_{m,M}(z)}dz\right|\le CN^{1-n}(16(R+r))^m.
\end{multline*}
Thus, we have
 \[\max_{R_j}|u(x)|\le A_1(B_1^m\max_{R_{j-1}}|u(x)|+q_1^m\max_{Q^{h}_{1}}|u(x)|),\] 
where $q_1=16(R+r)$ and $B_1=2+2Rr^{-1}$. Iterating this estimate $n-1$ times we obtain the desired estimate with $B=B_1^n$, $q=q_1B_1^{n-1}$ and $A=A_1^nB_1/(B_1-1)$. We have to check that $q<1$, indeed,
\[
q=16(R+r)2^{n-1}(R+r)^{n-1}r^{1-n}=2^{n+3}(R+r)^nr^{1-n}<2^{2n+3}r<1.\]    
\end{proof}

Finally, we prove Theorem \ref{th:m} for the case $R<2r<2^{-2n-3}$. We choose $m_0=\left[(\log M-\log \epsilon)(\log B-\log q)^{-1}\right]+1$. If $m_0<\sqrt{rh^{-1}}$ then applying Lemma \ref{l:three} with $m=m_0$, we obtain
\[\max_{Q_R^h}|u(x)|\le C\varepsilon^\alpha M^{1-\alpha}.\]
If $m_0\ge \sqrt{rh^{-1}}$ then we apply the Lemma with $m=[\sqrt{rh^{-1}}]$ and get
\[
\max_{Q_R^h}|u(x)|\le A_2 q^m M\le C\delta^{\sqrt{N}}M,\]
where $\delta=q^{\sqrt{r}}<1$.  

A standard argument with a chain of squares and iteration of the estimate gives the following.

\begin{corollary}
Let $\Omega$ be a connected domain in $\R^n$, $O$ be an open subset of $\Omega$, and $K\subset\Omega$ be a compact set. Then there exists $C,\alpha$ and $\delta<1$ and $N_0$ large enough such that for any $N\in\Z, N>N_0$, $h=N^{-1}$ and any $h$-harmonic function $u$ on $\Omega^h$ we have
\[
\max_{K^h} |u|\le C\left(\left(\frac{\max_{O^h}|u|}{\max_{\Omega^h}|u|}\right)^{\alpha}+\delta^{\sqrt{N}}\right)\max_{\Omega^h}|u|.\] 
\end{corollary}

\section{Concluding remarks}

 It is clear that a zero function on a cube can be extended non-trivially to a harmonic function on $\Z^n$. For example from a square $[-M,M]^2$ one may extend the function to a strip $[-M,M]\times\Z$ with arbitrary values at the points $(\pm M, y)$ with $|y|>M$ on the sides of the strip, then lay-wise the function is uniquely extended to a discrete harmonic function on $\Z^2$. The same argument works in higher dimensions.  We may also construct a discrete harmonic  polynomial continuation for any discrete harmonic function on the cube, see proposition below. Therefore a discrete version of three balls (or three cubes) theorem should have an error term that depend on the mesh-size of the lattice. It could be also reformulated in the following way. Given $r<R<1$ there exist $C,\alpha$ and a function $d(N)$ such that $d(N)\rightarrow 0$ as $N\rightarrow \infty$ and any discrete harmonic function $u$ on $[-N,N]\cap\Z^n$ satisfies the inequality
\[
\max_{|x|\le NR}|u|\le  C(\max_{|x|\le rN}|u|^\alpha\max_{|x|\le N}|u|^{1-\alpha}+d(N)\max_{|x|\le 1}|u|).\]
We have proved that one can take $d(N)=\delta^{\sqrt{N}}$.
If we consider a discrete harmonic function with values $\pm 1$ on the boundary of a rectangle $[-N,N]\times[-M,M]$, where $N>M$ and on the sides $\{\pm N\}\times[-N,N]$ and alternating signs, we see that $\max_{[-K,K]^2}\asymp b^{K-M}$ for some $b>1$, when $K>M$ (see \cite{GM}).
Thus the error term $d(N)$ can not go to zero faster than $q^N$ for some $q<1$. It would be interesting to find the correct asymptotic behavior of $d(N)$.

We finish this note by an elementary argument that shows how any discrete harmonic function on a cube can be extended to the lattice by a discrete harmonic polynomial. Note also that there are some discrete sets (like non-convex domains) for which  extension of a harmonic function could be impossible, see \cite{H}.  Let us define
\[K_{M,N}^{(n)}=\{(x_1,...x_n)\in\Z^n:\ |x_1|\le M, ..., |x_{n-1}|\le M, |x_n|\le N\},\]
we also denote $K_{N,N}^{(n)}$ by $K_N^{(n)}$ for simplicity.
We will prove the following
\begin{proposition}
\label{th:mvhi} If $f$ is a discrete harmonic function on $K_N^{(n)}$ then there exists a discrete harmonic polynomial $P$ on $\Z^n$ such that $\deg P\le 6N(n-1)+1$ and $f=P$ on $K_N^{(n)}$.
\end{proposition}

In dimension $2$ this proposition was proved in \cite{H}, where the estimate for the degree is precise, the details can be found in \cite{GPHD}. In higher dimensions we suggest a different argument, now the degree of the polynomial is not optimal.

We first claim that there exists a function $g$ discrete harmonic in $K_{3N,N}^{(n)}$ and such that $f=g$ in $K_N^{(n)}$. Further, values of $g$ on $K_N^{(n)}$ are determined by its values on two squares $K_{3N}^{(n-1)}\times\{-N+1,-N\}$ and we will be done if we show that there exists a discrete harmonic polynomial that coincides with $g$ on the set $K_{3N}^{(n-1)}\times\{-N+1,-N\}$.  By shifting the last variable we may instead consider the set $K_{3N}^{(n-1)}\times\{0,1\}$. Further we can find two polynomials $G_0$ and $G_1$ of $n-1$ variables such that $g(x,0)=G_0(x)$ and $g(x,1)=G_1(x)$ when $x\in \Z^{n-1}$ (this is standard multivariate polynomial interpolation on a grid), $G_0, G_1$ can be chosen of degree less than or equal to $6N(n-1)$. For the details we refer the reader to \cite[Chapter 4]{DB}. So we have reduced the Theorem to the following statement

\begin{lemma} Let $G_0$ and $G_1$ be polynomials  of $n-1$ variables with degree less than or equal to $M$. There exists a discrete harmonic polynomial $P$ on $\Z^n$ such that $P(x,0)=G_0(x)$ and $P(x,1)=G_1(x)$ for $x\in K_{3N}^{(n-1)}$ and the degree of $P$ is less than or equal to $M+1$.
\end{lemma}

\begin{proof}
We first find polynomials of one variable $q_j(t)=c_{j,j}t^j+c_{j,j-1}t^{j-1}+...+c_{j,1}t$ for $j\ge 2$ that satisfy \[\Delta_d q_j(t)=q_j(t+1)+q_j(t-1)-2q_j(t)=t^{j-2},\]   $q_j(0)=0$ and $q_j(1)=0$.
We also let $q_0=1$ and $q_1=t$.
Now we look for $P(x,x_n)$  in the form
\[P(x,x_n)=\sum_{j=0}^{M+1}q_j(x_n) Q_j(x),\]
and the conditions become $\Delta_d P(x,x_n)=0$, $P(x,0)=Q_0(x)=G_0(x)$, and $P(x,1)=Q_0(x)+Q_1(x)=G_1(x)$.
We have
\begin{multline*}
\Delta_d P(x,x_n)=\sum_{j=0}^{M+1}\left(\Delta_dq_j(x_n) Q_j(x)+ q_j(x_n)\Delta_d Q_j(x)\right)=\\
\sum_{j=2}^{M+1} x_n^{j-2} Q_j(x)+ \sum_{j=0}^{M+1}q_j(x_n)\Delta_d Q_j(x)=\\
\sum_{j=0}^{M-1} x_n^{j}\left(Q_{j+2}(x)+\sum_{k=j}^{M+1} c_{k,j}\Delta_d Q_k(x)\right)+\\x_n^{M}\sum^{M+1}_{M}c_{k,M}\Delta_dQ_{k}(x)+x_n^{M+1}c_{M+1,M+1}\Delta_dQ_{M+1}(x).
\end{multline*}
Now  $Q_0$ and $Q_1$ are given polynomials of degree less than or equal to $M$, it is sufficient to find sequence of polynomials $Q_2,..., Q_{M+1}$ such that the degree of $Q_j$ is less than or equal to $M-j+1$ and
\[
 Q_{j+2}(x)+\sum_{k=j}^{M+1} c_{k,j}\Delta_d Q_k=0 \quad j=0,...,M+1.\]
Now, by comparing coefficients we treat the equations as a linear system. Each polynomial $Q_j$ gives us unknowns (coefficients) and the total number of unknowns we get is
\[  \sum_{k=0}^{M-1}{n-1+k \choose n-1}.\]
The number of equations is exactly the same. The right-hand sides for this linear system come from given polynomials $Q_0$ and $Q_1$. To show that there is a solution, we have to show that $Q_0=Q_1=0$ gives only trivial solution $Q_2=...=Q_{M+1}=0$. If there exists a non-trivial solution, we choose polynomial $Q_l, \ l\geq2$  that has the highest degree. We have
\[
  Q_l=-\sum_{k=l-2}^{M+1} c_{k,l-2}\Delta_d Q_k,\]
that leads to a contradiction since the degree of the polynomial on the left-hand side is greater than the degree of the one on the right-hand side. 
\end{proof} 

 \subsection*{Acknowledgments}
  The authors are supported by the Research Council of Norway grants 185359/V30 and 213638. The second author would like to thank Center of Advanced Study at the Norwegian Academy of Science and Letters in Oslo where the work was completed.


\begin{thebibliography}{10}

\bibitem{A}  S. Agmon, {\it Unicit\'{e} et convexit\'{e} dans les probl\`{e}mes diff\'{e}rentiels,} S\'{e}minaire de Math\'{e}matiques Sup\'{e}rieures, no. 13 (\'{E}t\'{e}, 1965), Les Presses de l'Universit\'{e} de Montr\'{e}al, Montreal, 1966.
 

\bibitem{ARRV} G. Alessandrini, L. Rondi, E. Rosset, S. Vessella,{\it The stability for the Cauchy problem for elliptic equations,} Inverse Problems {\bf 25} (2009), 123004.

\bibitem{BHL} F. Boyer, F. Hubert, J. Le Rousseau, {\it Discrete Carleman estimates for elliptic operators in arbitrary dimension and applications,} SIAM J. Control Optim. {\bf 48} (2010), 5357--5397.

\bibitem{DB} G. Dahlquist, \AA. Bj\"{o}rck, Numerical methods in Scientific Computing, Vol. 1, SIAM, 2008.

\bibitem{EG} S. Ervedoza, F. de Gournay, {\it Uniform stability estimates for the discrete Calderon problems,} Inverse problems {\bf 27} (2011), 125012.

\bibitem{FM} R. S. Falk and P. B. Monk, {\it Logarithmic convexity for discrete harmonic functions and approximation of the Cauchy problem for Poisson's equation,} Math. Comp. {\bf 47} (1986), 135--149.

\bibitem{GL} N. Garofalo, F. Lin, {\it Monotonicity properties of variational integrals, $A_p$ weights and unique continuation,} Indiana Univ. Math. J. {\bf 35} (1986), 245--268.

\bibitem{G} M. Guadie, {\it Stability estimates for discrete harmonic functions on product domains,} Appl. Anal. Discrete Math. {\bf 7} (2013), 143--160.

\bibitem{GPHD} M. Gaudie, PhD thesis, NTNU, Trondheim, submitted June 2013.

\bibitem{GM} M. Gaudie, E. Malinnikova, {\it Stability and regularization for determining sets of discrete Laplacian,} to appear in Inverse Problems.

\bibitem{H} H. A. Heilbronn,{\it  On discrete harmonic functions,} Proc. Cambridge Philos. Soc. {\bf 45} (1949), 194--206.

\bibitem{KS} M. V. Klibanov, F. Santosa, {\it A computational quasi-reversibility method for Cauchy problems for Laplace's equation,}
SIAM J. Appl. Math. {\bf 51} (1991), 1653--1675.

\bibitem{KM} J. Korevaar, J. L. H. Meyers,{\it  Logarithmic convexity for supremum norms of harmonic functions,} Bull. London Math. Soc. {\bf 26} (1994), 353--362.

\bibitem{M} E. Malinnikova,  {\it Propagation of smallness for solutions of generalized Cauchy-Riemann systems,} Proc. Edinb. Math. Soc. {\bf 47} (2004), 191--204.

\bibitem{RHH} H.-J. Reinhardt, H. Han and D. N. H\'{a}o, {\it Stability and regularization of a discrete approximation to the Cauchy problem for Laplace's equation,} SIAM J. Numer. Anal. {\bf 36} (1999),  890--905.




\end{thebibliography}
\end{document}